\documentclass[a4paper,10pt]{article}

\usepackage[numbers,sort&compress]{natbib}

\usepackage[french,english]{babel}

\usepackage{pdfsync}

\usepackage{amsmath}                                                        
\numberwithin{equation}{section}
\usepackage{graphicx}                                                       
\usepackage{subfigure}
\usepackage{enumitem}                                                    
\usepackage{hyperref}
\usepackage{amssymb}                                                        
\usepackage[mathscr]{eucal}                                             
\usepackage[normalem]{ulem}                                                                 
\usepackage{pstricks}
\usepackage{rotating}
\usepackage[paperwidth=8.5in,paperheight=11in,top=1.00in, bottom=1.00in, left=1.00in, right=1.00in]{geometry}
\usepackage{mathtools}                                                      
\mathtoolsset{showonlyrefs=true}                                    
\usepackage{fixltx2e,amsmath}                                           
\MakeRobust{\eqref}

\usepackage{dsfont}

\usepackage{amsthm}                                                             
\allowdisplaybreaks

\theoremstyle{plain}
\newtheorem{theorem}{Theorem}
\numberwithin{theorem}{section}

\newtheorem{lemma}[theorem]{Lemma}                              
\newtheorem{proposition}[theorem]{Proposition}
\newtheorem{corollary}[theorem]{Corollary}
\theoremstyle{definition}
\newtheorem{definition}[theorem]{Definition}

\newtheorem{notation}[theorem]{Notation}
\newtheorem{remark}[theorem]{Remark}

\usepackage{color, soul}

\def \aa {\bar{\a}}
\def \av {{\a}}
\def \iot {{\kappa}}

\newcommand{\Canis}[1]{\mathbf{C}^{#1} }

\newcommand{\Linf}[2]{L^{\infty}_{#1}(\Canis{#2})}

\newcommand{\Cint}[2]{\mathbf{C}_{#1}^{#2} }

\newcommand{\St}[1]{\mathcal{S}_{#1}}

\def \ttau {t}

\def \a {{\alpha}}
\def \b {{\beta}}
\def \d {{\delta}}

\def \G {{\Gamma}}

\def \R {{\mathbb {R}}}
\def \N {{\mathbb {N}}}

\def \x {{\xi}}

\def \eps {{\varepsilon}}

\def \t {{\tau}}
\def \t {{\tau}}

\def \m {{\mu}}
\def \y {{\eta}}

\newcommand\dd{d}

\def \g {{\gamma}}

\def \phi {{\varphi}}

\def \tilde {\widetilde}

\def\p{\partial}

\def \a {{\alpha}}
\def \b {{\beta}}
\def \d {{\delta}}

\def \G {{\Gamma}}

\def \rn  {{\mathbb {R}}^{N}}
\def \R  {{\mathbb {R}}}
\def \x {{\xi}}
\def \g {{\gamma}}

\def \eps {{\varepsilon}}
\def \t {{\tau}}

\def \m {{\mu}}
\def \y {{\eta}}

\def \p {{\partial}}
\def \a {{\alpha}}

\def \d {{\delta}}

\def \a {{\alpha}}
\def \b {{\beta}}
\def \d {{\delta}}

\def \G {\Ga}

\def \Ga {{\Gamma}}

\def \R {{\mathbb {R}}}
\def \N {{\mathbb {N}}}

\def \x {{\xi}}

\def \eps {{\varepsilon}}

\def \t {{\tau}}
\def \t {{\tau}}

\def \m {{\mu}}
\def \y {{\eta}}

\def \g {{\gamma}}

\def \phi {{\varphi}}

\def \tilde {\widetilde}

\def \A {\mathcal{A}}

\def \rn {{\mathbb {R}}^{N}}

\def \L {\mathscr{L}}

\def \Lcal {{\A+Y}}
\def \Lcalp {{(Y+\A)}}

\def \à {{\`a }}
\def \è {{\`e }}
\def \ò {{\`o }}
\def \ù {{\`u }}

\newcommand\Ac{\mathscr{A}}

\newcommand{\<}{\langle}
\renewcommand{\>}{\rangle}
\renewcommand{\(}{\left(}
\renewcommand{\)}{\right)}

\def \sol {p} 
\def \gg {\mathbf{G}} 
\def \param {\mathbf{P}}
\def \C {\mathcal{C}}
\def \TT {T} 
\def \rr {q} 
\def \T {T}

\def \ee {\mathbf{e}}
\newcommand{\gauss}{\mathbf{G}} 

\makeatletter
\def\section{\@startsection {section}{1}{\z@}{3.25ex plus 1ex minus
 .2ex}{1.5ex plus .2ex}{\large\bf}}
\def\subsection{\@startsection{subsection}{2}{\z@}{3.25ex plus 1ex minus
 .2ex}{1.5ex plus .2ex}{\normalsize\bf}}
\@addtoreset{equation}{section} 

\begin{document}

\title{
Optimal Schauder estimates for kinetic Kolmogorov equations with time measurable coefficients}

\author{
Giacomo Lucertini
\thanks{Institut f\"ur Mathematik, Augsburg Universit\"at, Augsburg, 86159, Germany.
\textbf{e-mail}: giacomo.lucertini@uni-a.de}
\and Stefano Pagliarani
\thanks{Dipartimento di Matematica, Universit\`a di Bologna, Bologna, 40126, Italy.
\textbf{e-mail}: stefano.pagliarani9@unibo.it} \and Andrea Pascucci
\thanks{Dipartimento di Matematica, Universit\`a di Bologna, Bologna, 40126, Italy.
\textbf{e-mail}: andrea.pascucci@unibo.it}
}



\date{This version: \today}

\maketitle

\begin{abstract}
{We prove global Schauder estimates for kinetic Kolmogorov equations with coefficients that are
H\"older continuous in the spatial variables but only measurable in time. Compared to other
available results in the literature, our estimates are optimal in the sense that the inherent
H\"older spaces are the strongest possible under the given assumptions: in particular, under a
parabolic H\"ormander condition, we introduce H\"older norms defined in terms of the intrinsic
geometry that the operator induces on the space-time variables. The technique is based on the
existence and the regularity estimates of the fundamental solution of the equation. 
These results are essential for studying backward Kolmogorov equations
associated with kinetic-type diffusions, e.g. stochastic Langevin equation.}
\end{abstract}

\vspace{2pt}\noindent \textbf{MSC}: 35B65, 35K65, 60J60


%
%

\section{Introduction}
{In recent years}, several sharp Schauder estimates have been proved for Kolmogorov equations with
coefficients that are H\"older-continuous in the space-variables but only measurable in time.
In this article we prove global Schauder estimates which we claim to be {\it optimal}, meaning
that the inherent H\"older spaces are the strongest possible under the given assumptions on the
coefficients. In particular, our results
include and improve {some} known estimates in the framework 
of {non-divergence form} 
operators satisfying a {parabolic} H\"ormander condition.

A prototype example of the class under consideration is
\begin{equation}
  \frac{{\sigma^2(t,v,x)}}{2}\p_{vv}+v\p_{x}+\p_{t},\qquad (t,v,x)\in\R^{3},
\end{equation}
which is the backward Kolmogorov operator of the system of stochastic differential equations
\begin{equation}\label{ae3}
  \begin{cases}
    dV_{t}= {\sigma(t,V_t,X_t)} dW_{t},&\\
    dX_{t}=V_{t}dt,
  \end{cases}
\end{equation}
where $W$ is a real Brownian motion. 
The study of these models is motivated by several applications, including kinetic theory and
finance.  In the classical Langevin model, $(V,X)$ describes velocity and position of a particle
in the phase space and is a pilot example of more complex kinetic models (cf. \cite{MR1972000},
\cite{MR4275241}, \cite{MR4229202}). In mathematical finance, $(V,X)$ represents the log-price and
average processes utilized in modeling path-dependent financial derivatives, such as Asian options
(cf. \cite{MR1830951}, \cite{MR2791231}).

We now introduce the general class under consideration. 
Let $d,N\in\N$, with $d\le N$, be fixed throughout the paper. We denote by $(t,x)$ the point in
$\R\times\R^{N}$ and consider the second-order operator in non-divergence form
\begin{equation}\label{L}
 \L:=\Ac + Y
\end{equation}
where
\begin{align}\label{Lter}
 \Ac &:=\frac{1}{2}\sum_{i,j=1}^{d} a_{ij}(t,x)\p_{x_i x_j}+ \sum_{i=1}^{d} a_{i}(t,x)\p_{x_i}
 +a(t,x),\\
  Y&:=\p_{t}  +  \sum_{i,j=1}^{N}b_{ij}x_{j}\p_{x_{i}}= \p_t + \langle B x,\nabla\rangle,
\end{align}
and $B=(b_{ij})_{i,j=1,\cdots, N}$ is a constant matrix of dimension $N\times N$. The diffusion
part $\Ac$ is an elliptic operator on $\R^d$, while the drift (or transport) term $Y$ is a first
order differential operator
on $\R^{N+1}$. We impose two main structural {hypotheses}: 
\begin{itemize}
\item[{$\mathbf{(H.1)}$}] the matrix $A=\(a_{ij}\)_{i,j=1,\dots,d}$ is
symmetric and there exists a positive constant $\m$ such that 
\begin{equation}\label{aaa}
 {\m^{-1}}{|\y|^2} \leq \sum_{i,j=1}^{d}a_{ij}(t,x)\y_i \y_j \leq \m |\y|^2,\qquad x\in\R^N,\
 \y\in\R^{d},
\end{equation}
for almost every $t\in[0,\TT]$ where $T>0$ is fixed once for all;
\item[{$\mathbf{(H.2)}$}] 
{the following H\"ormander condition is satisfied:}
\begin{equation}\label{horcon}
 \text{rank } \text{Lie}(\partial_{x_1},\dots,\partial_{x_{d}},Y)= N+1.
\end{equation}
\end{itemize}
The focus of the paper is on the case $d<N$, which is $\L$ is 
fully degenerate, namely
no coercivity condition on $\R^N$ is satisfied. Condition \refeq{horcon} is also known as {\it parabolic} 
H\"ormander condition since the drift $Y$ plays a key role in the generation of the Lie algebra.

We consider the Cauchy problem
\begin{equation}\label{eq:Cauchy_prob}
 \begin{cases}
  \L u = f &\text{on }\St{T},\\
  u(T,\cdot)= g &\text{on }\R^N,
\end{cases}
\end{equation}
posed on the strip
\begin{equation}\label{aaa1}
  \St{T}:=\,]0,T[\times\R^{N}.
\end{equation}
Generally speaking, Schauder estimates give a bound of some H\"older norm of the solution $u$ in
terms of some (possibly different) H\"older norms of the data, namely the coefficients of $\L$,
the non-homogeneous term $f$ and the datum $g$. The ``strength'' of a Schauder estimate depends on
the norms involved and this is a sensitive issue in the theory of degenerate
PDEs: 
clearly, for a given H\"older norm on the solution $u$, the weaker the norms on the data, the
stronger the Schauder estimate; 
conversely, if the norms on the data are given, then the stronger the norm on $u$ the stronger the
Schauder estimate. Now, in the literature on degenerate Kolmogorov equations, we may recognize at
least two notions of H\"older norm as well as variants of them:
the so-called {\it anisotropic} 
and {\it intrinsic} 
norms, whose precise definitions are given in Section \ref{secHol}. {
Intuitively, the former norm takes into account the anisotropic behavior in space induced by the
underlying diffusion, but does not require any time-regularity. The intrinsic norm, by opposite,
is induced by the geometric properties of the differential operator $\L$ and takes into account
the regularity along the vector field $Y$ (and thus along the time variable). Therefore, the
intrinsic norm is stronger in the sense that it allows to see the full regularizing effect (in
both space and time) of the fundamental solution of $\L$.}

Roughly speaking, we may catalogue the known Schauder estimates for solutions to
\eqref{eq:Cauchy_prob} as follows:
\begin{itemize}
  \item {\it{anisotropic-to-anisotropic}}: 
 the anisotropic norms of the data bound the anisotropic norm of the  solution, as in \cite{MR1475774}, \cite{MR2136978}, \cite{MR2534181}, \cite{MR4334312}
 and \cite{MR4696151};
  \item {\it{intrinsic-to-intrinsic}}: the intrinsic norms of the data bound the intrinsic norm of the solution, as in \cite{MR1751429}, \cite{difpol}, {\cite{MR4072211}},  \cite{MR4275241} and  \cite{MR4405316};
  \item {\it{anisotropic-to-intrinsic}}: the anisotropic norms of the data bound the intrinsic norm of the
  solution. This class is stronger than the two above: 
only recently, partial results were proved 
in \cite{MR4924753} and \cite{MR4676647}.
\end{itemize}
{Our main result, Theorem \ref{th:main}, 
provides an optimal {\it{anisotropic-to-intrinsic}} global Schauder estimate which improves the
results in \cite{MR4676647} in a subtle but crucial way; moreover, in the particular case of $N=2d$, it is equivalent to the result in \cite{MR4924753}.
A more detailed comparison is provided in Section \ref{revlit}.}
For instance, the H\"older norm we adopt for the solution $u$ is strong enough to derive intrinsic
Taylor formulas,
{which are a basic tool to prove probabilistic results such as the {It\^o formula} (cf.
\cite{lanconelli2020local}), well-posedness results for SDEs (cf. \cite{MR4848654}), as well as
analytical (cf. \cite{MR4275241}) and asymptotic results (cf. \cite{MR3660883}).}

{Our estimate is global in that it holds all the way up to the boundary, with an explosion factor
that depends on the regularity of $g$, namely $(T-t)^{-\frac{2+\alpha-\beta}{2}}$: here $\alpha$
and $\beta$ represent the H\"older exponents of the solution $u$ and of the terminal datum $g$,
respectively. In particular, it is possible to recognize two limiting cases: $\beta=2+\alpha$, no
explosion close to the boundary; $\beta=0$ ($g$ only continuous), maximum explosion rate. As a
corollary, we obtain a sharp regularity estimate (Corollary \ref{cor:regY}) in the $Y$ direction,
at the boundary $t=T$.
Although this is an expected phenomenon, the quantitative characterization of the penalty term is
novel, to the best of our knowledge, for degenerate Kolmogorov operators in the context of
variable coefficients or of intrinsic H\"older spaces. We refer the reader to
\cite{MR4696275} for the case of constant diffusion coefficients and anisotropic
Besov-H\"older spaces}.
We also mention that intrinsic embedding theorems of Sobolev type were
recently proved in \cite{MR4444114} and \cite{pascucci2022sobolev},
and $L^p$ global Schauder estimates recently appeared in \cite{MR4444079}.

{Theorem \ref{th:main} comprises a well-posedness result for \eqref{eq:Cauchy_prob}. 
The proof of Theorem \ref{th:main} goes as follows. First we define a candidate solution $u$ via
Duhamel principle. After proving that it is actually a solution to \eqref{eq:Cauchy_prob}, we
prove the regularity estimates for the two convolution terms that constitute $u$, from which the
Schauder estimate follows. The proof critically relies on the recent results in
\cite{MR4660246}, where the existence of the fundamental solution of $\L$ was
established, together with optimal H\"older estimates, by means of a suitable modification of the
parametrix technique, already employed in \cite{MR1386366} and \cite{difpas} in the case of
intrinsic H\"older-continuous coefficients.}

\vspace{2pt}

{The rest of the article is organized as follows. Section \ref{sec:schauder} contains the main results and a detailed comparison with the related literature. Precisely, in Section \ref{secHol} we define both the anisotropic and intrinsic H\"older spaces; in Section \ref{sec:main_result} we state our main result, Theorem \ref{th:main},  and comment on it; Section \ref{revlit} contains the comparison with the literature. Section \ref{sec:proof} is entirely devoted to the proof Theorem \ref{th:main}.} 

\section{Schauder estimates}\label{sec:schauder}
{We begin by fixing some general notation that will be utilized throughout the rest of the
article.
Let $g:\R^N\rightarrow\R$. For any $i=1,\dots,N$, we denote by $\partial_i g(x)$ the partial
derivatives of $g$ with respect to $x_i$. We also denote by $\nabla_d$ the gradient operator
$(\partial_1, \dots, \partial_d)$ with respect to the first $d$ components and by $\nabla^2_d$ the
Hessian operator $(\partial_{ij})_{i,j=1,\dots,d}$ with respect to the first $d$ components. We
also consider the natural extensions of the operators $\nabla_d$ and $\nabla^2_d$ to functions
$f=f(t,x)$ defined on 
$\R\times \R^N$. We set the following basic functional spaces:
\begin{itemize}
\item[-] $C_b$, the set of bounded continuous functions $g:\R^N\rightarrow\R$, equipped with the norm
\begin{equation}
 \|g\|_{L^{\infty}}:=\sup\limits_{x\in\rn} |g(x)|;
\end{equation}
\item[-] $L^{\infty}_{\ttau}$, for $t>0$, the set of measurable functions 
$f$, defined on the strip $\St{\ttau}$ (cf. \eqref{aaa1}), such that the norm
\begin{equation}
 \| f \|_{L^{\infty}_{\ttau}}:= \sup_{(s,x)\in\St{\ttau}} |f(s,x)|
\end{equation}
is finite.
\end{itemize}

Finally, all the normed 
spaces in this article are defined for scalar valued functions and naturally extend to vector
valued functions by considering the sum of the norms of their single components.}

\subsection{H\"older spaces}\label{secHol} {In this section, we introduce the {\it anisotropic} and {\it
intrinsic} H\"older spaces that appear in the Schauder estimates for \eqref{eq:Cauchy_prob}.
Loosely speaking, the general idea behind the definition of these spaces is the following: 
\begin{itemize}
  \item anisotropic H\"older spaces are defined for functions of $x\in\R^{N}$, {\it assuming regularity in
all $N$ directions} w.r.t. an anisotropic distance that reflects the different time-scaling
properties of the underlying diffusion process. This distance is defined in term of an anisotropic
norm that assigns to each component of $x\in\R^{N}$ a different weight corresponding to the number
of commutators of
$\nabla_{d}
$ and $Y$ that are required to reach that direction. The definition then extends to functions defined on $\R^{N+1}$ 
by only requiring measurability and local boundedness with respect to the time variable.
  \item  intrinsic H\"older spaces are defined for functions of $(t,x)\in \R^{N+1}$ that are assumed to be
anisotropically H\"older continuous, in the sense above, uniformly in time. Additional H\"older
regularity in the direction of the drift $Y$ is also assumed. By means of the H\"ormander
condition, it is then possible to infer H\"older regularity jointly with respect to all variables.
\end{itemize}}

Let us first recall that the parabolic H\"ormander condition \eqref{horcon} is equivalent to the
well-known Kalman rank condition for controllability in linear systems theory (cf., for instance,
\cite{MR2791231}). Also, it was shown in \cite{lanpol} that \eqref{horcon} is equivalent to $B$
having the block-form
\begin{equation}\label{B}
  B=\begin{pmatrix}
 \ast & \ast & \cdots & \ast & \ast \\
 B_1 & \ast &\cdots& \ast & \ast \\
 0 & B_2 &\cdots& \ast& \ast \\
 \vdots & \vdots &\ddots& \vdots&\vdots \\
 0 & 0 &\cdots& B_{\rr}& \ast
  \end{pmatrix}
\end{equation}
where the $*$-blocks are arbitrary and $B_j$ is a $(d_{j-1}\times d_j)$-matrix of rank $d_j$ with
\begin{equation}
  d\equiv d_{0}\geq d_1\geq\dots\geq d_{\rr}\geq1,\qquad \sum_{i=0}^{\rr} d_i=N.
\end{equation}
The block decomposition \eqref{B} induces naturally an anisotropic norm for $x\in\rn$ defined as
\begin{equation}\label{e7}
 |x|_B:=\sum_{j=0}^{\rr}\, \sum_{i=\bar{d}_{j-1}+1}^{\bar{d}_j} |x_{i}|^{\frac{1}{2j+1}}, \qquad
 {\bar{d}_j:= \sum_{k=0}^j d_k}, \qquad d_{-1}:=0.
\end{equation}

\begin{definition}[\bf Anisotropic H\"older spaces]\label{defanis} Let $\av\in\,]0,3]$.
\begin{itemize}
\item The anisotropic H\"older {norms} on $\R^{N}$ are
defined recursively as 
\begin{align}
 \|g\|_{\Canis{\av}}:= \begin{cases} 
 \|g\|_{L^{\infty}}+ \sup\limits_{x,y\in\rn} \frac{|g(x) - g(y)|}{|x - y|^{\av}_B}, &\av\in\,]0,1],\\
 \label{and1}
 \|g\|_{L^{\infty}}+\|\nabla_{\!d}g\|_{\Canis{\av-1}}+
 \sup\limits_{(x,\y)\in\R^{N}\times\R^{N-d}}
 \frac{|g(x+(0,\y)) - g(x)|}{|(0,\y)|^{\av}_B},
 &\av\in\,]1,3].
 \end{cases}
 \end{align}
 {
 We denote by $\Canis{\alpha}$ the set of functions $g:\rn\rightarrow\R$ such that the norm $ \|g\|_{\Canis{\av}}$ is finite. Set also $\Canis{0}:=C_b$.}
 \item  {For $\ttau>0$, 
 the anisotropic H\"older norms on $\St{\ttau}$ are
\begin{align}
\text{(weighted)}\quad \| f \|_{\Linf{\ttau,\gamma}{\alpha}} &:= \sup_{s\in\,]0,\ttau[} \Big(
(\ttau-s)^{\gamma} \| f(s,\cdot) \|_{\Canis{\alpha}} \Big),\quad \gamma\in[0,1[,\qquad\\
\text{(non-weighted)}\quad \| f \|_{\Linf{\ttau}{\alpha}} &:= \| f
\|_{\Linf{\ttau,0}{\alpha}}.\qquad\qquad\qquad
\end{align}
We denote by $\Linf{\ttau,\gamma}{\alpha}$ and $\Linf{\ttau}{\alpha}$ the set of measurable
functions $f:\St{\ttau}\rightarrow\R$ such that the norms $ \| f \|_{\Linf{\ttau,\gamma}{\alpha}}
$ and $ \| f \|_{\Linf{\ttau}{\alpha}} $, respectively, are finite. }
\end{itemize}
\end{definition}

{Before introducing the intrinsic H\"older spaces, we recall the following}
\begin{definition}
[\bf {Lie derivative}] \label{def:Lie_ae_bis}
{For $\ttau>0$ and $\av \in\,]0,2]$ we set}
\begin{align}
  \| f \|_{C^{\av}_{Y,\ttau}}&:= \sup_{(\tau,x)\in\St{\ttau}}\sup_{s\in\,]0,\ttau[\smallsetminus\{\t\}}\frac{\left|f(s,e^{(s-\tau) B}x) -
 f(\tau,x)\right|}{|s-\tau|^{\frac{\av}{2}}}.
\end{align}
Moreover, we say that {\it $f$ is a.e. Lie-differentiable along $Y$ on $\St{\ttau}$} if there
exists $F\in L_{\text{\rm loc}}^{1}(]0,\ttau[;C_b(\rn))$ such that
 \begin{equation}
 f(s,e^{(s - \tau) B}x) = f(\tau,x) + \int_{\tau}^{s}  F(r, e^{(r - \tau) B} x )d r, \qquad
 (\tau,x)\in\St{\ttau},\  s\in\,]0,t[.
\end{equation}
In that case, we set $Yf=F$ and call it an {\it a.e. Lie derivative of $f$ on $\St{\ttau}$}.
\end{definition}

\begin{definition}[\bf Intrinsic H\"older spaces]\label{def:Lie_ae}
{Let $\ttau>0$. The {\it intrinsic H\"older norms on $\St{\ttau}$} are defined recursively as:}
\begin{align}
  \|f\|_{\Cint{\ttau}{\av}}:=\begin{cases}
  \|f\|_{\Linf{\ttau}{\alpha}}
  +\|f\|_{C^{\av}_{Y,\ttau}},&\hspace{0pt}\quad \av\in\,]0,1],\\
 \|f\|_{\Linf{\ttau}{\alpha}}+\|\nabla_{\!d}f\|_{\Cint{\ttau}{\av-1}}+\|f\|_{C^{\av}_{Y,\ttau}},&\hspace{0pt}\quad \av\in\,]1,2],\\
\|f\|_{\Linf{\ttau}{\alpha}}+\|\nabla_{\!d}f\|_{\Cint{\ttau}{\av-1}}+\| {Yf}
\|_{\Linf{\ttau}{\av-2}}, &\hspace{0pt}\quad \av\in\,]2,3].
  \end{cases}
\end{align}
{For $\av\in\,]0,3]$, we denote by $\Cint{\ttau}{\av}$ the set of functions
$f:\St{\ttau}\rightarrow\R$ such that the norm $ \|f\|_{\Cint{\ttau}{\av}}$ is finite.}
\end{definition}

\begin{remark}[\bf Intrinsic vs anisotropic spaces]\label{rem:C0alpha_intrins_b}
Obviously, the intrinsic space $\Cint{\ttau}{\av}$ is strictly included in the anisotropic space
$\Linf{\ttau}{\av}$. {Note that, for $\alpha\in\,]0,1]$, the addition in the
$\Cint{\ttau}{\av}$-norm of the term $\|f\|_{C^{\av}_{Y,\ttau}}$ yields H\"older regularity
jointly in the time and space variables: in particular, it is standard to show that if $f\in
\Cint{\ttau}{\av}$ then
\begin{equation}
|f(\tau,x) - f(s,y) | \leq C \| f \|_{\Cint{\ttau}{\av}} \big(  |\tau - s|^{\frac{\alpha}{2}} + |x
- e^{(\tau - s)} y|^{\alpha}_B  \big), \qquad (\tau,x),(s,y) \in \St{\ttau}.
\end{equation}}
\end{remark}

\begin{remark}[\bf Intrinsic Taylor formula]\label{rem:taylor}
\label{rem:C0alpha_intrins} For $\alpha\in\,]0,1]$, the intrinsic spaces $\Cint{\ttau}{\av}$ and
$\Cint{\ttau}{1+\av}$ are equal to those in \cite{pagpaspig}. However,
$\Cint{\ttau}{2+\av}$ 
is slightly weaker than the one in \cite{pagpaspig} in that the Lie derivative $Yu$ is not
required to be in $\Cint{\ttau}{\av}$ but only in $\Linf{\ttau}{\av}$. This difference is dictated
by our assumptions on the coefficients that are merely measurable in the temporal variable: so,
for a solution $u$ of \eqref{eq:Cauchy_prob}, one may expect that $Yu$ exists in the strong sense
but, in general, is not more than bounded in the $Y$-direction. {Despite this,
$\Cint{\ttau}{2+\av}$ in Definition \ref{def:Lie_ae} is still strong enough to prove the following
{\it intrinsic Taylor formula} as in \cite[Theorem 2.10]{pagpaspig}:
%
if {$u\in\Cint{\ttau}{2+\av}$} then 
\begin{equation}
| u(\tau,x) - T_2u\big(s,y; x - e^{(\tau-s)B} y\big)  |  \lesssim  \| u \|_{\Cint{\ttau}{2+\av}}
\big( |\tau-s| + | x - e^{(\tau-s)B} y |^{2+\alpha}_B \big), \qquad (\tau,x),(s,y)\in \St{\ttau},
\label{eq:taylor_formula_intrins}
\end{equation}
where $T_2 u$ is the second order intrinsic Taylor polynomial
\begin{equation}
T_2 u (s,y; z  ) = u(s,y) +  \sum_{i=1}^d z_i \,  \partial_i u(s,y) + \frac{1}{2} \sum_{i,j=1}^d
 z_i z_j\,  \partial_{ij} u(s,y).
\end{equation}
Furthermore, by adding the term $(\tau-s)Y u (s,y)$ to $T_2 u$, the estimate can be improved by
obtaining a term of order $o(|\tau-s|)$, as $\t-s\to0$, in place of $|\tau-s|$ in
\eqref{eq:taylor_formula_intrins}.}

It is worth noting that, for $u$ in the anisotropic space $\Linf{t}{\av}$, {\it estimate
\eqref{eq:taylor_formula_intrins} generally holds only for $s=\tau$}: this is the best result that
one can deduce from Schauder estimates of anisotropic-to-anisotropic type.
\end{remark}


\subsection{Main result}\label{sec:main_result}
We remark that if $u\in \Linf{T}{2+\av}$ or $u\in \Cint{T}{2+\av}$ then the derivatives
$\p_{x_{i}}u,\p_{x_{i}x_{j}}u$ exist in the classical sense only for $1\le i,j\le d$. Indeed, $u$
is not regular enough to support even the first-order derivatives $\p_{x_{i}}$ for $d<i\le N$ and
the full gradient $\nabla u$ appearing in the drift of $\L$ must be interpreted in a suitable way.
{Thus, in accordance with the intrinsic H\"older spaces defined above, 
we employ the following natural definition of solution to the kinetic equation.}

{\begin{definition}[\bf Strong Lie solution]\label{def:solution_eq} \label{solint}
Let $f\in L_{\text{\rm loc}}^{1}(]0,T[;C_{b}(\rn))$. 
A solution to equation
\begin{equation}\label{ae1}
  \A u+ Yu = f\ \text{on }\mathcal{S}_{T}
\end{equation}
is a continuous function $u:\St{T}\rightarrow\R$ such that there exist
$\nabla_{d}u,\nabla^2_{d}u\in
L_{\text{\rm loc}}^{1}(]0,T[;C_{b}(\rn))$ and $Yu=f-\A u$ 
in the sense of Definition \ref{def:Lie_ae_bis}.

For $g\in C(\rn)$, a solution to the Cauchy problem \eqref{eq:Cauchy_prob} is a solution $u$ to
\eqref{ae1}, which can be extended continuously to $]0,T]\times \rn$ with $u(T,\cdot)= g$.
\end{definition} }

\begin{theorem}\label{th:main}
Let 
assumptions 
{\bf (H.1)} and {\bf (H.2)} be in force and assume also the coefficients $a_{ij},a_{i},a$ of $\L$
be in $\Linf{\TT}{\aa}$ for some $\aa\in\,]0,1]$. Then, for any $\av\in\,]0,\aa[$, $g\in
\Canis{\beta}$ with $\beta\in [0,2+\av]$ and
$f\in\Linf{T,\gamma}{\av}$ with $\gamma\in[0,1[$, there exists a unique 
bounded {strong Lie} solution $u$ to the Cauchy problem \eqref{eq:Cauchy_prob}. Furthermore, we have 
\begin{equation}\label{eq:main_res}
 \|{u}\|_{\Cint{t}{2+\av}} \leq C \left((T-t)^{-\frac{2+\av - \beta}{2}} \| g \|_{\Canis{\beta}} +
( T-t )^{-{\gamma}} \| f \|_{\Linf{T,\gamma}{\av}}\right), \qquad  0< t < T,
\end{equation}
where $C$ is a positive constant which depends only on $\TT,B,\aa,\av,\beta,\gamma$ and on the
norms {$\Linf{\TT}{\aa}$} of the coefficients of $\A$. {In particular, $C$ does not depend on
$t$.}
\end{theorem}
We illustrate the Schauder estimate through particular instances; by linearity, we can treat the
cases $g=0$ and $f=0$ separately:
\begin{itemize}
  \item {\bf [Case $g=0$]} Estimate \eqref{eq:main_res} reads as
\begin{equation}\label{eq:main_res_g_zero}
 \|{u}\|_{\Cint{t}{2+\av}} \leq C ( T-t )^{- {\gamma}} \| f \|_{\Linf{T,\gamma}{\av}},\qquad 0<t<T.
\end{equation}
In particular, if $\gamma = 0$ then $f$ is bounded and \eqref{eq:main_res_g_zero} holds true up to
the boundary $t=T$:
\begin{equation}\label{eq:main_res_g_zero_b}
 \|{u}\|_{\Cint{T}{2+\av}} \leq C\,
 \| f \|_{\Linf{T}{\av}}.
\end{equation}

  \item {\bf [Case $f=0$]} We have two extreme cases:
\begin{itemize}
  \item[$\diamond$] If $\b=0$, that is $g$ is only bounded and continuous, then
  the solution has the same {explosion behavior} of the fundamental solution as $t\to T^{-}$ (cf. \cite{MR4660246}), which
  is
    $$\|{u}\|_{\Cint{t}{2+\av}} \leq  C (T-t)^{-\frac{2+\av}{2}} \| g \|_{L^{\infty}}.$$
  \item[$\diamond$] If $\b=2+\a$ then the solution is $(2+\av)$-H\"older continuous up to the boundary
\begin{equation}\label{eq:main_res_eps_zero_unif}
   \|{u}\|_{\Cint{t}{2+\av}} \leq C \| g \|_{\Canis{2+\a}}.
\end{equation}
\end{itemize}
\end{itemize}
Estimate \eqref{eq:main_res_eps_zero_unif} entails a regularity result along $Y$ at the boundary
$t=T$, which is reported in the following
\begin{corollary}\label{cor:regY}
Let the assumptions of Theorem \ref{th:main} be in force with $\g=0$. 
The solution $u$ to the Cauchy problem \eqref{eq:Cauchy_prob} satisfies
\begin{equation}
\| u(t, e^{(t-T)B} \cdot) - u(T,\cdot)   \|_{L^{\infty}} \leq C (T-t)^{\frac{\beta}{2}}, \qquad
0<t < T,
\end{equation}
if $\beta \in\,]0,2]$.
\end{corollary}

\begin{remark}
Recall that (see Remark \ref{rem:taylor}) Theorem  \ref{th:main} does not imply, under the assumptions therein, that the Lie derivative $Yu$ is jointly continuous in time and space variables. However, under the additional assumption that $f$ and  $a_{ij},a_{i},a$ are continuous on $\St{T}$, it can be seen by Definitions \ref{def:Lie_ae} and \ref{solint}, together with Remark \ref{rem:C0alpha_intrins_b}, that $Yu$ turns out to be continuous on $S_T$. In particular, $u$ is Lie differentiable along $Y$ everywhere on $\St{T}$, and thus equation \eqref{ae1} is satisfied pointwise everywhere on $\St{T}$.
\end{remark}

\subsection{Comparison with the literature}\label{revlit}


Global {{\it anisotropic-to-anisotropic}} Schauder estimates were proved by Lunardi
\cite{MR1475774}, Lorenzi \cite{MR2136978}, Priola \cite{MR2534181}, Menozzi \cite{MR4334312}
under different assumptions on the coefficients,
stronger or equivalent to ours. These estimates read as follows: if $u$ is a 
solution of the Cauchy problem \eqref{eq:Cauchy_prob} 
then
\begin{equation}\label{ae4}
   \|{u}\|_{\Linf{T}{2+\av}} \leq C \left(\| g \|_{\Canis{2+\av}} +
   \| f \|_{\Linf{T}{\av}}\right).
\end{equation}
Estimate \eqref{ae4} is similar to \eqref{eq:main_res} for the particular choice $\b=2$ and
$\g=0$; however, by Remark \ref{rem:C0alpha_intrins_b}, estimate \eqref{ae4} is weaker than
\eqref{eq:main_res} due to the strict inclusion of intrinsic into anisotropic spaces. Moreover,
the above mentioned papers assume a smooth datum, $g\in \Canis{2+\av}$, missing the smoothing
effect of the equation, which is well-known in the uniformly parabolic case (cf.
\cite{MR1184023}). On the other hand, the class of equations considered by Menozzi
\cite{MR4334312} allows for more general drift terms. {Zhang \cite{MR4696275} proved general
estimates in the context of Besov-Holder spaces, which, at order two read as
\begin{equation}\label{eq:zhang_besov}
\|{u}\|_{\Linf{T}{2+\av}} \leq C (T-t)^{-\frac{2+\av - \beta}{2}} \| g \|_{\Canis{\beta}}.
\end{equation}
Once more, this estimate is proved for the anisotropic H\"older norm and thus only catches the
smoothing effect of the kernel with respect to the spatial variables. Also, \eqref{eq:zhang_besov}
is proved in the case of $\Ac$ with constant coefficients.}

{Global intrinsic-to-intrinsic Schauder estimates} were obtained by Imbert and Mouhot
\cite{MR4275241} for $B$ as in \eqref{B1}, assuming coefficients in {$\Cint{T}{\av}$}. As noticed
in \cite{MR4924753}, the results in \cite{MR4275241} do not cover the case of some elementary
smooth functions, for instance, $d=1$ and $f(t,x_{1},x_{2})=\sin x_{2}$.
We also mention that {\it interior estimates} were obtained by Manfredini
\cite{MR1751429}, Di Francesco and Polidoro
\cite{difpol}, Henderson and Snelson \cite{MR4072211} and very recently by Polidoro, Rebucci and
Stroffolini \cite{MR4405316} for operators with Dini continuous coefficients.

Global anisotropic-to-intrinsic Schauder estimates 
were obtained by Biagi and Bramanti \cite{MR4676647} under the additional assumption that
the $\ast$-blocks in \eqref{B} are null and therefore $\L$ is
homogeneous w.r.t. a suitable family of dilations: if $u$ is a 
solution to $\L u=f$ then
\begin{align}\label{ae6}
 {\|\nabla^2_{d}  u\|_{L^{\infty}_{T}(\Canis{\av})} }+\| Y u\|_{L^{\infty}_{T}(\Canis{\av})}
  +
  {\|\nabla_{d} u\|_{\Cint{T}{\av}}  + \|u\|_{\Cint{T}{\av}} } \leq C\left(\| u\|_{L^\infty_T}+ \|f\|_{L^{\infty}_{T}(\Canis{\av})}\right).
\end{align}
Estimate \eqref{ae6} is weaker than \eqref{eq:main_res} because
 the norm on the l.h.s. of \eqref{ae6} is smaller than the norm in $\Cint{T}{2+\av}$: {the H\"older semi-norms $\|\nabla_d u\|_{C^{1+\av}_{Y,T}}$ and $\|\nabla^2_d u\|_{C^{\av}_{Y,T}}$ along the direction $Y$ are missing.
} Moreover, the optimal regularity for the second derivatives $\nabla^2_d u$ is obtained only
\emph{locally}.

The closest results to ours were recently obtained by Dong and Yastrzhembskiy
\cite{MR4924753} for the Langevin operator in $\R^{2d+1}$, that is for $B$ in \eqref{B} in
the particular form
\begin{equation}\label{B1}
 B=\begin{pmatrix}
    0 & 0 \\
    I_{d} & 0 \
  \end{pmatrix},
\end{equation}
where $I_{d}$ is the $d\times d$ identity matrix and for the Cauchy problem with null-terminal
datum, $g= 0$. The techniques in \cite{MR4924753} are based on a kernel free approach inspired by
Campanato's ideas. {Although the proof is completely different, the estimates are equivalent to
ours in this particular case. To be more precise, the main result in Theorem 1.6. of
\cite{MR4924753} includes, for the solution to the Cauchy problem, H\"older continuity of order
$(\alpha+2)/3$ in the degenerate direction $x$, which is not contained in \eqref{eq:main_res}.
However, the same regularity in the degenerate direction is recovered in our setting through the
intrinsic Taylor formula \eqref{eq:taylor_formula_intrins}: this is possible because
\eqref{eq:main_res} includes the optimal regularity along $Y$ of the first order derivatives
(which is not included in Theorem 1.6. in \cite{MR4924753}, but it is later proved in Corollary
1.8).
}

\section{Proof of Theorem \ref{th:main}}\label{sec:proof}

The proof of Theorem \ref{th:main} relies on the results of \cite{MR4660246} where a
fundamental solution to the operator $\L=\Lcal$ was constructed in the form
\begin{equation}\label{eq:fundamental_solution}
p(t,x;T,y) = \param(t,x;T,y) + \Phi(t,x;T,y), \qquad 0< t <{T},  \quad  x,y \in\R^N,
\end{equation}
where:
\begin{itemize}
\item the function $\param$ is a so-called \emph{parametrix}, which is defined as
\begin{equation}\label{eq:def_parametrix}
 \param(t,x;s,y):=\G^{(s,y)}(t,x;s,y),\qquad   0< t<s <T, \quad x,y \in\R^N ,
\end{equation}
where, for $(\tau,v)\in\mathcal{S}_{T}$, we set
\begin{equation}\label{eq:param}
\G^{(\tau,v)}(t,x;s,y):=\gauss\big(\C ^{(\tau,v)}(t,s),y-e^{(s-t)B}x\big),\qquad    0< t<s < T,
\quad x,y \in\R^N,
\end{equation}
with
\begin{equation}
\gg(\mathcal{C},z):=
\frac{1}{\sqrt{(2\pi)^{N}\det \mathcal{C}}}e^{-\frac{1}{2}\<\mathcal{C}^{-1}z,z\>} ,
\end{equation}
and
\begin{align}\label{cao}
\C^{(\tau,v)}(t,s)&:= \int_{t}^{s} e^{(s-r) B} A^{(\tau,{v})} (r) e^{(s-r) B^*}d r , \\
\label{aao}
A^{(\tau,v)}(r)&:=\begin{pmatrix} A_0(r,e^{(r-\tau)B}v)& 0 \\ 0 & 0\end{pmatrix},\qquad
A_{0}=\(a_{ij}\)_{i,j=1,\dots,d};
\end{align}
\item the function $ \Phi(t,x;T,y)$ is a remainder enjoying suitable regularity estimates, which are recalled in Proposition \ref{lem:deriv_theta} below.
\end{itemize}
The strategy of our proof is to define a candidate solution to the Cauchy problem
\eqref{eq:Cauchy_prob} in the form
\begin{equation}\label{eq:repres_u}
 u(t,x):=\int_{\R^N} p(t,x;T,\y)g(\y)d\y   - \int_t^T\int_{\R^N} p(t,x;\t,\y)f(\t,\y)d\y d\t , \qquad (t,x)\in \St{T} ,
\end{equation}
and prove that $u$:
{(i) is the unique bounded solution to \eqref{eq:Cauchy_prob} 
and (ii)
satisfies the estimate \eqref{eq:main_res}.}

Note that $u$ in \eqref{eq:repres_u} can be written as

\begin{equation}\label{eq:repres_u_gzero}
 u(t,x)= 
 V_{g}(t,x) - V_{\param,f}(t,x) - V_{\Phi,f}(t,x)   , \qquad (t,x)\in \St{T} 
\end{equation}
with
\begin{equation}
V_{g}(t,x) = \int_{\R^N} p(t,x;T,\y)g(\y)d\y,
\end{equation}
and
\begin{equation}
V_{\param,f}(t,x): = \int_t^T\int_{\R^N} \param(t,x;\t,\y)f(\t,\y)d\y d\t, \qquad V_{\Phi,f}(t,x):
=  \int_t^T\int_{\R^N} \Phi(t,x;\t,\y)f(\t,\y)d\y d\t   .
\end{equation}
We now prove our main result, Theorem \ref{th:main}. The proof is based on the sharp regularity
estimates for $V_{g}$, $V_{\param,f}$ and $V_{\Phi,f}$ contained in Propositions
\ref{prop:holder_bounds_Vg} and \ref{prop:holder_bounds_V_param_Phi}.
\begin{proof}[Proof of Theorem \ref{th:main} (well-posedness of \eqref{eq:Cauchy_prob})]
{The uniqueness of the solution follows from the maximum principle for parabolic equations: 
we refer to \cite[Chapter 6]{MR494490} for a detailed proof}.

We prove that $u$ as defined in \eqref{eq:repres_u} is a solution to the Cauchy problem
\eqref{eq:Cauchy_prob} {in the sense of Definition \ref{solint}.} The facts that 
{$\nabla_{d}u,\nabla^2_{d}u$} have the required regularity on $\St{T}$, and that $u$ can be
extended continuously to the closure of $\St{T}$ in a way that $u(T,\cdot)\equiv g$, are
straightforward consequences of the estimates of Propositions
\ref{prop:holder_bounds_V_param_Phi}-\ref{prop:holder_bounds_Vg} and of the Dirac delta property
of $p$, that is
\begin{equation}
\lim_{\substack{(t,x)\to(\T,y)\\ t<\T}} \int_{\R^N}\sol(t,x;\T,\y)\phi(\y)d\y=\phi(y),\qquad \phi\in C_{b}(\rn).
\end{equation}
To prove that 
{$Yu=f-\A u$ 
in the sense of Definition \ref{def:Lie_ae_bis}}, we can consider separately, by linearity, two
cases. {We state here, once and for all, that all the applications of Fubini's theorem throughout
this proof are justified by the estimates of Propositions \ref{prop:holder_bounds_V_param_Phi} and
\ref{prop:holder_bounds_Vg}.}

\vspace{2pt}

\noindent\emph{Case $f\equiv 0$.}  We have
\begin{align}
u(s,e^{(s-t)B}x) - u(t,x) &= V_{g}(s,e^{(s-t)B}x) - V_{g}(t,x) = \int_{\R^N} \big(
p(s,e^{(s-t)B}x;T,\y) - p(t,x;T,\y)  \big) g(\y) \dd \y \intertext{($\Lcalp p(\cdot,\cdot;T,\y)
=0$ on $\St{T}$ in the sense of Definition \ref{def:solution_eq})} & = - \int_{\R^N} \int_t^s  \Ac
p(r,e^{(r-t)B}x;T,\y) \dd r\,  g(\y) \dd \y =  - \int_t^s  \Ac u(r,e^{(r-t)B}x)   \dd r,
\end{align}
where, in the last equality, we employed Fubini's theorem and
\eqref{eq:represent_deriv_potent_param_bis_g}-\eqref{eq:represent_deriv_potent_param_g} to move
the operator $\Ac$ out of the integral on $\rn$. \vspace{2pt}

\noindent\emph{Case $g\equiv 0$.} We have
\begin{align}
u(s,e^{(s-t)B}x) - u(t,x) &=  \underbrace{- \int_s^T \int_{\R^N} \big( p(s,e^{(s-t)B}x;\tau,\y ) -
p(t,x;\tau,\y) \big)f(\tau,\y) \dd \y \dd \tau}_{=:I}\\ &\quad  +\int_t^s \int_{\R^N}
p(t,x;\tau,\y) f(\tau,\y) \dd \y \dd \tau.
\end{align}
As $\Lcalp p(\cdot,\cdot;\tau,\y) =0$ on $\St{\tau}$ in the sense of Definition
\ref{def:solution_eq}, for any $\tau\in]s,T[$, we have
\begin{equation}
p(s,e^{(s-t)B}x;\tau,\y ) - p(t,x;\tau,\y) = - \int_{t}^s \Ac p (r,e^{(r-t)B}x;\tau,\y) \dd r.
\end{equation}
Therefore, by Fubini's theorem, we have
\begin{align}
I &= \int_t^s \int_s^T \int_{\R^N}\Ac p (r,e^{(r-t)B}x;\tau,\y) f(\tau,\y) \dd \y \dd \tau \dd r
\\ & =  \underbrace{\int_t^s  \int_r^T \int_{\R^N} \Ac p (r,e^{(r-t)B}x;\tau,\y) f(\tau,\y) \dd \y
\dd \tau \dd r }_{=: I_1} - \underbrace{\int_t^s \int_r^s \int_{\R^N}\Ac p (r,e^{(r-t)B}x;\tau,\y)
f(\tau,\y) \dd \y \dd \tau \dd r }_{=: I_2}.
\end{align}
By moving the operator $\Ac$ out of the integral in $\dd \y \dd \tau$, we obtain
\begin{equation}
I_1 = - \int_t^s \Ac u (r,e^{(r-t)B}x)  \dd r,
\end{equation}
and thus, to show $\Lcalp u ={f}$ on $\St{T}$, we need to prove that
\begin{equation}\label{eq:I2_result}
I_2 = - \int_t^s f (r,e^{(r-t)B}x)  \dd r + \int_t^s \int_{\R^N}  p(t,x;\tau,\y) f(\tau,\y) \dd \y
\dd \tau .
\end{equation}
By applying Fubini's theorem we obtain
\begin{equation}
I_2 =  \int_t^s \underbrace{ \int_t^{\tau} \int_{\R^N} \Ac p (r,e^{(r-t)B}x;\tau,\y) f(\tau,\y) \dd \y\, \dd r }_{=:J(\tau)}  \dd \tau .
\end{equation}
Now, for any $\eps \in\,]0,\tau - t[$, we can write
\begin{equation}
J(\tau) = \underbrace{\int_{\R^N} \int_t^{\tau - \eps} \Ac p (r,e^{(r-t)B}x;\tau,\y) f(\tau,\y)\dd
r \dd \y }_{=:J^{\eps}_1(\tau)}  + \underbrace{ \int_{\tau - \eps}^{\tau} \int_{\R^N} \Ac p
(r,e^{(r-t)B}x;\tau,\y) f(\tau,\y) \dd \y \, \dd r }_{=:J^{\eps}_2(\tau)}.
\end{equation}
As $\Lcalp p(\cdot,\cdot;\tau,\y) =0$ on $\St{\tau}$, we obtain
\begin{equation}
J^{\eps}_1(\tau) =-  \int_{\R^N} \big( p (\tau-\eps,e^{(\tau-\eps-t)B}x;\tau,\y) - p (t,x;\tau,\y)
\big) f(\tau,\y) \dd \y,
\end{equation}
and since $f(\tau,\cdot)$ is bounded and continuous, the Dirac delta property of $p$ yields
\begin{equation}\label{eq:J1eps_lim}
J^{\eps}_1(\tau) \to -  f (\tau,e^{(\tau-t)B}x)   +  \int_{\R^N}  p(t,x;\tau,\y) f(\tau,\y) \dd \y
,\qquad \text{as } \eps \to 0^+.
\end{equation}
%
{We finally consider $J^{\eps}_2$. By the estimates of {Lemma \ref{lemm:potf} and of Proposition \ref{lem:deriv_theta}},
together with \eqref{eq:fundamental_solution}, we have
\begin{equation}
\left|\int_{\R^N} \Ac p(r,e^{(r-t)B}x;\tau,\y) f(\tau,\y) \dd \y\right|\leq
\frac{C}{(T-\t)^{\g}(\t-r)^{\frac{2-\av}{2}}}\|f\|_{\Linf{T,\gamma}{\av}},
\end{equation}
that is integrable in $dr$ on $[t,\t]$. Then, Lebesgue dominated convergence theorem yields}
\begin{equation}
J^{\eps}_2(\tau) \to 0  ,\qquad \text{as } \eps \to 0^+.
\end{equation}
This and \eqref{eq:J1eps_lim} 
yield \eqref{eq:I2_result}.
\end{proof}
\begin{proof}[Proof of Theorem \ref{th:main} (estimate \eqref{eq:main_res})]
As $u$ is a solution to \eqref{eq:Cauchy_prob}, $Yu =  f - \Ac u$ is a Lie derivative of $u$ on
$\St{T}$. Therefore, we have
{
\begin{equation}
  \| u \|_{\Cint{t}{2+\av}}   \leq\|u\|_{\Linf{t}{2+\av}}+  \|\nabla_d u\|_{\Cint{t}{1+\av}} +  \| \Ac u\|_{\Linf{t}{\av}} +   \| f\|_{\Linf{t}{\av}}   .
\end{equation}
Now we recall that, 
  for $\alpha\in]0,1]$, the intrinsic spaces $\Cint{\ttau}{\av}, \Cint{\ttau}{1+\av}$ are exactly equivalent to those in \cite{pagpaspig}. This is a consequence of the intrinsic Taylor formula of Theorem 2.10 in the latter reference, with $n=0,1$. In particular, we have
\begin{equation}
\|\nabla_d u\|_{\Cint{t}{1+\av}} \leq C (        \|\nabla_d u\|_{L^{\infty}_t}+  \| \nabla^2_d
u\|_{L^{\infty}_t}+  \|\nabla_d u\|_{C^{1+\av}_{Y,t}} +\|\nabla^2_d u\|_{C^{\av}_{Y,t}}
+\|\nabla^2_d u\|_{C^{\av}_{\nabla_d,t}}     ),
\end{equation}
where
\begin{equation}
\|\nabla^2_d
u\|_{C^{\av}_{\nabla_d,t}}:={\sup_{i,j=1,\dots,d}\Big(\sup_{(s,x)\in\St{t}}\sup_{h\in\R^{d}}
\frac{|(\nabla^2_d u)_{ij}(s,x+(h,0)) - (\nabla^2_d u)_{ij}(s,x)|}{|h|^{\av}}\Big)}.
\end{equation}
Therefore, the estimates of Propositions \ref{prop:holder_bounds_V_param_Phi} and \ref{prop:holder_bounds_Vg} 
 yield
\begin{equation}\label{eq:bound_N1P}
\|\nabla_d u\|_{\Cint{t}{1+\av}}  \leq C \Big(  (T-t)^{-\frac{2+\av - \beta }{2}} \| g
\|_{\Canis{\beta}} +
( T-t )^{-{\gamma}} \| f \|_{\Linf{T,\gamma}{\av}} \Big). 
\end{equation}
{To obtain the same estimate for $\|u\|_{\Linf{t}{2+\av}}$, it is enough to prove it for
\begin{equation}
\sup_{s\in]0,t[}\ \sup\limits_{(x,\y)\in\R^{N}\times\R^{N-d}}
 \frac{|u(s,x+(0,\y)) - u(s,x)|}{|(0,\y)|^{2+\av}_B},
\end{equation}
which can be done by proceeding as in the proofs of \eqref{fcampidritti} and
\eqref{fcampidritti_g}: we omit the details for brevity.}
Furthermore, 
by the regularity assumptions on the coefficients of $\Ac$, 
we have
\begin{equation}
  \| \Ac u\|_{\Linf{t}{\av}} \leq C( \|u\|_{\Linf{t}{\av}}+  \|\nabla_d u\|_{\Cint{t}{1+\av}}) .
\end{equation}
Finally, we have $\| f\|_{\Linf{t}{\av}} \leq  (T-t)^{-\gamma} \| f \|_{\Linf{T,\gamma}{\av}} $
and thus \eqref{eq:main_res}.}

\vspace{2pt}

\end{proof}

The rest of the section is devoted to proving the regularity estimates employed in the proof of
Theorem \ref{th:main}. Hereafter, we denote, indistinctly, by $C$ any positive constant depending
at most on $T,B,\m, \aa,\av,\beta,\gamma$ and on the $\Linf{T}{\aa}$ norms of the coefficients of
$\A$. We also introduce the following

{\begin{notation} For any $f=f(t,x;T,y)$ and  $i =1,\dots, N$, we set
\begin{equation}
\p_{i} f(t,x;T,y):= \p_{x_{i}}  f(t,x;T,y),
\end{equation}
and we adopt analogous notations for the higher-order derivatives. Thus, $\p_{i}$ always denotes a
derivative with respect to the first set of space variables. 
{Some caution is necessary when considering the composition of $f$ with a given function $F=F(x)$:
$\p_{i} f\big(t,F(x);T,y\big)$ denotes the derivative $\p_{z_i} f(t,z;T,y)|_{z=F(x)}$, and
similarly for higher order derivatives.} We also denote by $e_k$ the $k$-th element of the
canonical basis of $\rn$.
\end{notation}}

{Finally, we note that $\param,\Phi$, and their non-degenerate derivatives are continuous by
definition (see \cite[Section 2]{MR4660246}). Therefore, $p, V_g, V_{\param,f}, V_{\Phi,f}$, and
their non-degenerate derivatives are well-defined continuous functions for every $0<t<T$ and
$x,y\in\R^N$. }

\subsection{Preliminaries results}
We first recall the useful result \cite[Lemma {3.3}]{MR4660246}:
\begin{lemma}\label{lem:deriv_param_sol1}
Let $(t,y)\in\mathcal{S}_{T}$. Then, for any $i=1,\dots, d$, the function $u:=\p_{i}
\param(\cdot,\cdot;t,y) $ is a strong Lie solution to the equation
\begin{equation}
\A^{{(t,y)}} u+ Yu = - {\sum_{j=1}^{d+d_1} b_{ji} \p_{j}\param (\cdot,\cdot;t,y)}  \ \text{on
}\mathcal{S}_{t},
\end{equation}
in the sense of {Definition \ref{def:solution_eq}.} 
\end{lemma}

In order to state the next preliminary lemma, we fix the following
\begin{notation}
Let $\iot=(\iot_1,\dots,\iot_N)\in\mathbb{N}_0^N$ be a multi-index. {Recalling 
\eqref{e7},} 
we define the {\it$B$-length} of $\iot$ as
\begin{equation}
[\iot]_B:= \sum_{j=0}^{\rr}(2j+1)\sum_{i=\bar{d}_{j-1}+1}^{\bar{d}_j}\iot_{i}. 
\end{equation}
{Moreover, we define the so-called \it{homogeneous dimension} of $\R^N$ with respect to the quasi-norm
$|\cdot|_B$ as
\begin{equation}\label{ae11}
 Q=\sum_{i=0}^{\rr} (2i+1) d_i .
\end{equation}}
\end{notation}
We have the following potential estimates, whose proof is identical to the one of
\cite[Proposition B.2]{MR4660246}.
\begin{lemma}\label{lemm:potf}
For any $\iot\in\N^N_0$ with $[\iot]_B\le 4$, we have
\begin{align}
\left|\int_{\R^N}\param(t,x;\t,\y)f(\t,\y) d\y\right|\leq&
\frac{C}{(T-\t)^{\g}}\|f\|_{\Linf{T,\gamma}{\av}} 
 \label{eq:potf}\\
\left|\int_{\R^N} \p^{\iot}_x \param(t,x;\t,\y)f(\t,\y) d\y\right|\leq&
\frac{C}{(T-\t)^{\g}(\t-t)^{\frac{[\iot]_B-\av}{2}}}\|f\|_{\Linf{T,\gamma}{\av}}\label{eq:derpotf}
\end{align}
{for every $0<t<T$, $x\in\rn$, and for a.e. $\t\in]t,T[$}.
\end{lemma}

We now recall the estimates for $\param$ and $\Phi$ proved in \cite[Propositions 2.7-3.2 and Lemma
3.4]{MR4660246}, which are stated in terms of the Gaussian density $\G^{\delta}$
defined, for any $\delta>0$, by
\begin{equation}\label{gaussian}
\G^{\d}(t,x;{\t},y):=\gg\big(\d\mathcal{C}({\t}-t),y-e^{({\t}-t)B}x\big), \qquad 0\le t<{\t}\le T,
\ x,y\in\rn,
\end{equation}
with
\begin{equation}\label{eq:Ct}
\mathcal{C}(t)= \int_{0}^{t}
e^{(t-\t) B} 
\begin{pmatrix} I_{d} & 0 \\
0 & 0\end{pmatrix} e^{(t-\t) B^*}d\t.
\end{equation}
{We also recall that $\m$ is fixed in {$\mathbf{(H.1)}$} as the ellipticity constant for the matrix $A$.}

\begin{proposition}\label{lem:deriv_theta_param}
For any $i,j,k = 1,\dots,d$, we have 
\begin{align}\label{eq:est_param}
\left|{\param}(t,x;{\t},y)\right| &\leq C \, \G^{2\m}(t,x;{\t},y), \\
\label{eq:est_der_first_param} \left|\p_{i}{\param}(t,x;{\t},y)\right| &\leq
\frac{C}{({\t}-t)^{\frac{1}{2}}} \G^{2\m}(t,x;{\t},y),
\\ \left|\p_{i j}{\param}(t,x;{\t},y)\right| &\leq \frac{C}{{\t}-t} \G^{2\m}(t,x;{\t},y), \label{eq:est_der_second_param}
\end{align}
and
\begin{align}\label{dercurveinteg_param}
\big|\p_{i}\param(s,e^{(s-t)B}x;{\t},y)-\p_{i}\param(t,x;{\t},y)\big|
&\leq C \frac{(s-t)^{\frac{1+\av}{2}} ({\t}-t)^{\frac{Q}{2}}}{({\t}-s)^{\frac{Q+1+\av}{2}}}
\G^{2\m}(t,x;{\t},y),\\
\big|\p_{ij}\param(s,e^{(s-t)B}x;{\t},y)-\p_{ij}\param(t,x;{\t},y)\big| 
&\leq C \frac{(s-t)^{\frac{\av}{2}} ({\t}-t)^{\frac{Q}{2}}}{({\t}-s)^{\frac{Q+2+\av}{2}}}
\G^{2\m}(t,x;{\t},y),  \label{dercurveinteg2_param}\\ \left|\p_{ij}\param(t,x  +h \mathbf{e}_k
;{\t},y)-\p_{ij}\param(t,x;{\t},y)\right|
&\leq C |h|^{\av} {\frac{\G^{2\m}(t,x +h \mathbf{e}_k ;{\t},y)+\G^{2\m}(t,x;{\t},y)}{({\t}-t)^{\frac{2 + \av }{2}}} }, 
\label{dercampidritti_param}
\end{align}
for any $0{<} t<s<{\t}{<} T$, $x\in\rn$ and $h\in\R$.
\end{proposition}

\begin{proposition}\label{lem:deriv_theta}
For any $i,j,k = 1,\dots,d$, we have 
\begin{align}\label{eq:est_Phi}
\left|{\Phi}(t,x;{\t},y)\right| &\leq C ({\t}-t)^{\frac{\aa}{2}} \G^{2\m}(t,x;{\t},y),
\\ \label{eq:est_der_first_Phi} \left|\p_{i}{\Phi}(t,x;{\t},y)\right| &\leq
\frac{C}{({\t}-t)^{\frac{1-\aa}{2}}} \G^{2\m}(t,x;{\t},y),
\\ \left|\p_{i j}{\Phi}(t,x;{\t},y)\right| &\leq \frac{C}{({\t}-t)^{\frac{2-\aa}{2}}} \G^{2\m}(t,x;{\t},y), \label{eq:est_der_second_Phi}
\end{align}
and
\begin{align}\label{dercurveinteg}
\big|\p_{i}\Phi(s,e^{(s-t)B}x;{\t},y)-\p_{i}\Phi(t,x;{\t},y)\big|
&\leq C \frac{(s-t)^{\frac{1+\av}{2}} ({\t}-t)^{\frac{Q}{2}}}{({\t}-s)^{\frac{Q+2+\av-\aa}{2}}}
\G^{2\m}(t,x;{\t},y),\\
\big|\p_{ij}\Phi(s,e^{(s-t)B}x;{\t},y)-\p_{ij}\Phi(t,x;{\t},y)\big| 
&\leq C \frac{(s-t)^{\frac{\av}{2}} ({\t}-t)^{\frac{Q}{2}}}{({\t}-s)^{\frac{Q+2+\av-\aa}{2}}}
\frac{}{} \G^{2\m}(t,x;{\t},y), \label{dercurveinteg2}\\ \left|\p_{ij}\Phi(t,x  +h \mathbf{e}_k
;{\t},y)-\p_{ij}\Phi(t,x;{\t},y)\right|
&\leq C |h|^{\av} {\frac{\G^{2\m}(t,x +h \mathbf{e}_k ;{\t},y)+\G^{2\m}(t,x;{\t},y)}{({\t}-t)^{\frac{2- {(\av-\aa)} }{2}}} }, 
\label{dercampidritti}
\end{align}
for any $0{<} t<s<{\t}{<} T$, $x\in\rn$ and $h\in\R$.
\end{proposition}

By the Lemma \ref{lemm:potf} and Proposition \ref{lem:deriv_theta}, we have the following, direct,
\begin{proposition}\label{prop:repres_derivatives_convol}
For ${\bf F} = \param, \Phi$, we have 
\begin{align}\label{eq:represent_deriv_potent_param_bis}
\nabla_d V_{{\bf F},f}(t,x) &=  \int_t^T\int_{\R^N} 
\nabla_d {\bf F}(t,x;\t,\y)f(\t,\y)d\y d\t,\\ \label{eq:represent_deriv_potent_param}
\nabla^2_d V_{{\bf F},f}(t,x)& =  \int_t^T\int_{\R^N} 
\nabla^2_d {\bf F}(t,x;\t,\y)f(\t,\y)d\y d\t,
\end{align}
for any $0<t<T$ and $x\in\rn$.
\end{proposition}
The following identities directly stem from the {boundedness assumption on $g$ 
}
and from Propositions \ref{lem:deriv_theta_param} and \ref{lem:deriv_theta}. 
\begin{proposition}\label{prop:repr_der_poten_g}
We have
\begin{align}\label{eq:represent_deriv_potent_param_bis_g}
\nabla_d V_{g}(t,x) &=  \int_{\R^N} 
\nabla_d\,  p(t,x;T,y)g(y)d y ,\\ \label{eq:represent_deriv_potent_param_g}
\nabla^2_d V_{g}(t,x) &=  \int_{\R^N} 
\nabla^2_d\,  p(t,x;T,y)g(y)d y,
\end{align}
for any $0<t<T$, $x\in\rn$.
\end{proposition}

In the sequel we will make use of the special functions ${}_2F_1$ and $\mathbf{B}$, which denote
the Gaussian hypergeometric function and the incomplete Beta function, respectively.
\begin{remark}\label{rem:iper}
We recall the following known properties. For any $\gamma\in[0,1)$ and $ \alpha\in (0,1)$ we have:
\begin{itemize}
\item[(a)] ${}_2F_1\big( \frac{\av - 1}{2},\gamma;\frac{1+\av }{2};\cdot\big)$ is bounded on $[0,1/2]$;
\item[(b)] 
there exists $\kappa=\kappa({\alpha,\gamma})>0$ such that
\begin{equation}
\mathbf{B}(x,\a/2 , 1-\g) \leq \kappa\, x^{\alpha/2}, \qquad x\in [0,1].
\end{equation}
\end{itemize}
\end{remark}
%
%
%

\subsection{H\"older estimates for 
$V_{\param,f}$ and $V_{\Phi,f}$}

In this section we prove the following H\"older estimates for 
$V_{\param,f}$ and $V_{\Phi,f}$, on which the proof of Theorem \ref{th:main} relies.
\begin{proposition}\label{prop:holder_bounds_V_param_Phi}
For ${\bf F} = \param, \Phi$, and for any $i,j,k=1,\dots,d$, we have
\begin{align}\label{eq:bound}
  |  V_{{\bf F},f}(t,x)  | &\leq C (T-t)^{-\gamma+1} \|f\|_{\Linf{T,\gamma}{\av}}, \\
\label{eq:bound_der_first}
  | \p_{i} V_{{\bf F},f}(t,x)  | &\leq C (T-t)^{-\gamma+\frac{1+\av}{2}} \|f\|_{\Linf{T,\gamma}{\av}}, \\
| \p_{i j} V_{{\bf F},f}(t,x)  |&\leq C {(T-t)^{-\gamma +\frac{\av}{2}}
\|f\|_{\Linf{T,\gamma}{\av}} },  \label{eq:bound_der_second}
\end{align}
and
\begin{align}\label{fcampidritti}
|\p_{i j} V_{{\bf F},f} (t,x+h\ee_k) - \p_{i j} V_{{\bf F},f}(t,x)  |
&\leq C |h|^{\av} {(T-t)^{-\gamma}  \|f\|_{\Linf{T,\gamma}{\av}} }, 
 \\ \label{reg_integral_Vparam}
  |\p_{i} V_{{\bf F},f} \big(s,e^{(s-t)B}x\big) - \p_{i} V_{{\bf F},f}(t,x)  |
&\leq C(s-t)^{\frac{1+\av}{2}} (T-s)^{-\gamma} \|f\|_{\Linf{T,\gamma}{\av}}, 
\\ \label{reg_integral_Vparam_bis}
|\p_{i j} V_{{\bf F},f} \big(s,e^{(s-t)B}x\big) - \p_{i j} V_{{\bf F},f}(t,x)  |
&\leq C(s-t)^{\frac{\av}{2}} (T-s)^{-\gamma} \|f\|_{\Linf{T,\gamma}{\av}}, 
\end{align}
for any $0<t<s<T$, $x\in\rn$ and $h\in \R$. 
\end{proposition}

\begin{proof}[Proof of Proposition \ref{prop:holder_bounds_V_param_Phi} for ${\bf F} = \param$.]





Estimates \eqref{eq:bound}-\eqref{eq:bound_der_first}-\eqref{eq:bound_der_second} are a
straightforward consequence of estimates \eqref{eq:potf}, \eqref{eq:derpotf} with $\partial^{\iot}
= \partial_i$, and \eqref{eq:derpotf} with $\partial^{\iot} = \partial_{i j}$, respectively.

We now fix $0<t<T$, $x\in\rn$ and prove \eqref{fcampidritti} in two separate cases.\\
\vspace{2pt}\emph{Case {$2 h^2 \leq T-t$}.} 
We define
\begin{align}
I(\t):=\int_{\R^N} \big(\p_{i j}\param(t,x+h\ee_k;\t,y)-\p_{i j}\param(t,x;\t,y)\big) f(\t,y)dy
\end{align}
so that
\begin{align}
|\p_{i j} V_{\param,f} (t,x+h\ee_k) - \p_{i j} V_{\param,f}(t,x)  |
=\underbrace{\int_{t+h^2}^T I(\t)d\t }_{=:I_1} + \underbrace{\int_t^{t+h^2} I(\t)d\t }_{=:I_2}.
\end{align}
We consider $I_1$. 
By the mean-value theorem, there exists a real $\bar h$ with $|\bar h|\leq |h|$ such that
\begin{equation}
\left| \p_{ij}\param (t,x + h \ee_k ;\t,\y) - \p_{ij}\param(t,x;\t,\y)\right|=|h| \left|
\p_{ijk}\param (t,x +\bar h \mathbf{e}_k ;\t,\y)\right|.
\end{equation}
Therefore, by the estimate {\eqref{eq:derpotf} with $\p^{\iot}_x = \p_{i j k}$} 
we have 
\begin{equation}\label{eq:est_I_1}
|I_1| \leq C \int_{t+h^2}^T\frac{|h|}{(T-\t)^\g({\t-t})^{\frac{3-\av}{2}}}d\t \
\|f\|_{\Linf{T,\g}{\av}} .
\end{equation}
Now, a direct computation yields
\begin{equation}
\int_{t+h^2}^T\frac{1}{(T-\t)^\g({\t-t})^{\frac{3-\av}{2}}}d\t = \frac{\Gamma_{E}\Big(\frac{\av -
1}{2}\Big)}{(T-t)^{\gamma}} \Bigg(  (T-t)^{\frac{\av - 1}{2}}
\frac{\Gamma_{E}(1-\g)}{\Gamma_{E}\Big(\frac{1+\av }{2}-\g\Big)}    - |h|^{\av - 1} {}_2F_1\Big(
\frac{\av - 1}{2},\gamma;\frac{1+\av }{2};\frac{h^2}{T-t}\Big)   \Bigg),
\end{equation}
{where $\Gamma_{E}$ and ${}_2F_1$ denote, respectively, the Euler Gamma and the Gaussian}
hypergeometric functions. {This, together with \eqref{eq:est_I_1}, ${2 h^2 \leq T-t}$ and Remark
\ref{rem:iper}-(a), proves
\begin{equation}
|I_1|\leq C |h|^{\av} \|f\|_{\Linf{T,\g}{\av}}  (T-t)^{-\gamma}.
\end{equation}}

We now consider $I_2$. 
By employing triangular inequality and estimate {\eqref{eq:derpotf} with $\p^{\iot}_x = \p_{i j}$} 
we obtain
\begin{equation}\label{eq:est_I2}
|I_2|\leq \int_t^{t+h^2}\frac{C}{(T-\tau)^{\gamma}(\t-t)^{\frac{2-\av}{2}}} d\t \ \|f\|_{\Linf{T,\gamma}{\av}}.
\end{equation}
A direct computation yields
\begin{equation}
 \int_t^{t+h^2}\frac{1}{(T-\tau)^{\gamma}(\t-t)^{\frac{2-\av}{2}}} d\t = (T-t)^{\frac{\av}{2} - \gamma} \mathbf{B}\Big( \frac{h^2}{T-t} ,  \frac{\av}{2} , 1 -  \g  \Big) ,
\end{equation}
{where $\mathbf{B}$ denotes the incomplete Beta function.} This, together with \eqref{eq:est_I2}
and {Remark \ref{rem:iper}-(b)}, proves
\begin{equation}
|I_2|\leq C |h|^{\av} \|f\|_{\Linf{T,\g}{\av}}  (T-t)^{-\gamma},
\end{equation}
and thus \eqref{fcampidritti} when ${2 h^2 \leq T-t}$.
\\
\vspace{2pt}\emph{Case ${2 h^2 > T-t}$.} 
By employing the triangular inequality and estimate {\eqref{eq:derpotf} with $\p^{\iot}_x = \p_{i j}$} 
we obtain
\begin{align}
|\p_{i j} V_{\param,f} (t,x+h\ee_k) - \p_{i j} V_{\param,f}(t,x)  |
& \leq C \int_t^{T}\frac{\|f\|_{\Linf{T,\gamma}{\av}}}{(T-\tau)^{\gamma}(\t-t)^{\frac{2-\av}{2}}}
d\t  \leq  C  \, \|f\|_{\Linf{T,\gamma}{\av}} (T-t)^{\av/2-\g} \intertext{(as  ${T-t \leq 2
h^2}$)} & \leq  C  \, |h|^{\av} \|f\|_{\Linf{T,\gamma}{\av}} (T-t)^{-\g} .
\end{align}

We now prove \eqref{reg_integral_Vparam}.
%
%
%
By adding and subtracting, we have
\begin{align}
  \p_{i} V_{\param,f} \big(s,e^{(s-t)B}x\big) - \p_{i} V_{\param,f}(t,x)
&=\underbrace{\int_{s}^T\int_{\R^N}
\Big(\p_{i}\param(s,e^{(s-t)B}x;\t,\y)-\p_{i}\param(t,x;\t,\y)\Big) f(\t,\y)d\y d\t }_{=: I} \\
&\quad -\underbrace{\int_{t}^{s}\int_{\R^N} \p_{i}\param(t,x;\t,\y) f(\t,\y)d\y d\t}_{=: L}.\qquad
\label{eq:estim_split}
\end{align}
Estimate {\eqref{eq:derpotf} with $\p^{\iot}_x = \p_{i}$} 
yields
\begin{equation}
|L| \leq C\int_t^s\frac{1}{(T-\tau)^{\gamma}(\t-t)^{\frac{1-\av}{2}}} d\t
\|f\|_{\Linf{T,\gamma}{\av}}
\leq  C (s-t)^{\frac{1+\av}{2}} (T-s)^{-\g}\|f\|_{\Linf{T,\gamma}{\av}} .
\end{equation}
We now prove
\begin{equation}\label{eq:bound_I}
|I|\leq C (T-s)^{-\g} (s-t)^{\frac{1+\av}{2}}\|f\|_{\Linf{T,\gamma}{\av}}.
\end{equation}
Set ${\delta}:=s-t$ and consider, once more, two separate cases.
\\\vspace{2pt}\emph{Case ${2 {\delta}<T-s}$.} We split the integral
\begin{align}
I
&= \underbrace{\int_{s}^{s+{\delta}} 
\int_{\R^N} \Big(\p_{i}\param(s,e^{(s-t)B}x;\t,\y)-\p_{i}\param(t,x;\t,\y)\Big)
f(\t,\y)d\y}_{=:H_1} \\
&\quad +\underbrace{\int_{s+{\delta}}^{T} 
\int_{\R^N}  \Big(\p_{i}\param(s,e^{(s-t)B}x;\t,\y)-\p_{i}\param(t,x;\t,\y)\Big)   f(\t,\y)d\y}_{=:H_2}.
\end{align}
We estimate $H_1$. By the triangular inequality and {\eqref{eq:derpotf} with $\p^{\iot}_x = \p_{i}$} 
we have
\begin{align}
|H_1|&\leq C \int_{s}^{s+{\delta}}
\frac{1}{(T-\tau)^{\gamma}}\(\frac{1}{(\t-s)^{\frac{1-\av}{2}}}+
\frac{1}{(\t-t)^{\frac{1-\av}{2}}}\)d\t \,  \|f\|_{\Linf{T,\gamma}{\av}} \intertext{(since
$\t-s<\t-t$)} &\leq C \int_{s}^{s+{\delta}} \frac{1}{(T-\tau)^{\gamma}(\t-s)^{\frac{1-\av}{2}}}
d\t \, \|f\|_{\Linf{T,\gamma}{\av}}
\intertext{(by a direct computation)} &\leq C (T-s)^{-\g} (T-s)^{\frac{1+\av}{2}}
B\Big(\frac{{\delta}}{T-s},\frac{1+\av}{2},1-\gamma\Big)\|f\|_{\Linf{T,\gamma}{\av}}.
\end{align}
Therefore by {Remark \ref{rem:iper}-(b)} we have
\begin{equation}\label{eq:bound_H1}
|H_1|\leq C (T-s)^{-\g} {\delta}^{\frac{1+\av}{2}}\|f\|_{\Linf{T,\gamma}{\av}}.
\end{equation}
We now consider $H_2$.
By Lemma \ref{lem:deriv_param_sol1} and {by Fubini's theorem} 
we have
\begin{equation}
H_2 = - \int_{s+{\delta}}^T \int_t^s \int_{\R^N} \bigg(  \big( \A^{(\t,\y)}
\p_{i}\param\big)(r,e^{(r-t)B}x;\t,\y) + \sum_{j=1}^{d+d_1}
b_{ji}\p_{j}\param(r,e^{(r-t)B}x;\t,\y) \bigg) f(\t,\y)d\y dr d\t.
\end{equation}
Therefore, the estimates {\eqref{eq:derpotf} with $[{\iot}]_B = 3$} 
and {the regularity assumptions on the coefficients} 
yield
\begin{align}
|H_2|&\leq C \|f\|_{\Linf{T,\gamma}{\a}} \int_{s+{\delta}}^T \int_t^s
\frac{1}{(T-\t)^\g({\t-r})^{\frac{3-\av}{2}}}  dr d\t \\ &\leq C h \|f\|_{\Linf{T,\gamma}{\a}}
\int_{s+{\delta}}^T \frac{1}{(T-\t)^\g({\t-s})^{\frac{3-\av}{2}}} d\t \intertext{(by direct
computation)} &= C \delta \|f\|_{\Linf{T,\gamma}{\a}} \frac{\Gamma_{E}\Big(\frac{\av -
1}{2}\Big)}{(T-s)^{\gamma}} \Bigg(  (T-s)^{\frac{\av - 1}{2}}
\frac{\Gamma_{E}(1-\g)}{\Gamma_{E}\Big(\frac{1+\av }{2}-\g\Big)}    - h^{\frac{\av -
1}{2}}{}_2F_1\Big( \frac{\av - 1}{2},\gamma;\frac{1+\av }{2};\frac{{\delta}}{T-s}\Big) \Bigg).
\label{eq:bound_I2} \intertext{(as ${2 {\delta} \leq T-s}$ and by Remark \ref{rem:iper}-(a))}
&\leq C (T-s)^{-\g} {\delta}^{\frac{1+\av}{2}}\|f\|_{\Linf{T,\gamma}{\av}},
\end{align}
which, 
together with \eqref{eq:bound_H1}, yields \eqref{eq:bound_I}.
\\
\vspace{2pt}\emph{Case ${2{\delta} \geq T-s}$.}
By triangular inequality, and by {\eqref{eq:derpotf} with $\p^{\iot}_x = \p_{i}$}, we obtain 
\begin{align}
|I|&\leq C \int_{s}^{T} \frac{1}{(T-\tau)^{\gamma}}\(\frac{1}{(\t-s)^{\frac{1-\av}{2}}}+
\frac{1}{(\t-t)^{\frac{1-\av}{2}}}\) \|f\|_{\Linf{T,\gamma}{\av}} d\t \intertext{(since
$\t-s<\t-t$)} &\leq C \int_{s}^{T} \frac{1}{(T-\tau)^{\gamma}(\t-s)^{\frac{1-\av}{2}}}  d\t
\|f\|_{\Linf{T,\gamma}{\av}}\\ &\leq C
{(T-s)^{\frac{1+\av}{2}-\gamma}}\|f\|_{\Linf{T,\gamma}{\av}} \intertext{(as ${T-s\leq 2
{\delta}}$)} &\leq C {\delta}^{\frac{1+\av}{2}}(T-s)^{-\gamma}\|f\|_{\Linf{T,\gamma}{\av}},
\end{align}
which is \eqref{eq:bound_I}. The proof of \eqref{reg_integral_Vparam_bis} is completely analogous,
and thus is omitted for brevity.
\end{proof}

\begin{proof}[Proof of Proposition \ref{prop:holder_bounds_V_param_Phi} for ${\bf F} = \Phi$.]
Estimates \eqref{eq:bound}-\eqref{eq:bound_der_first}-\eqref{eq:bound_der_second} can be easily obtained from estimates \eqref{eq:est_Phi}-\eqref{eq:est_der_first_Phi}-\eqref{eq:est_der_second_Phi}, respectively. The details are omitted for sake of brevity. 

By \eqref{dercampidritti} we obtain 
\begin{align}
|\p_{i j} V_{\Phi,f} (t,x+h\ee_k) - \p_{i j} V_{\Phi,f}(t,x)  |
&\leq C \int_t^T\int_{\R^N}  |h|^{{\av}} \frac{\G^{2\m}(t,x +h \mathbf{e}_k
;\t,\y)+\G^{2\m}(t,x;\t,\y)}{(T-\t)^{\g}(\t-t)^{\frac{2- (\aa-\av) }{2}}}d\y \, d\t \,
\|f\|_{\Linf{T,\gamma}{\av}} \intertext{{(integrating in $\y$)}} &{\leq C} |h|^{{\av}} \int_t^T
\frac{1}{(T-\t)^{\g}(\t-t)^{\frac{2- (\aa-\av) }{2}}} d\t \|f\|_{\Linf{T,\gamma}{\av}}\\ &\leq C
|h|^{\av} (T-t)^{-\g+\frac{\aa-\av}{2}} \|f\|_{\Linf{T,\gamma}{\av}},
\end{align}
which proves \eqref{fcampidritti}.

We now prove \eqref{reg_integral_Vparam}. By adding and subtracting, we have
\begin{equation}
\begin{aligned}
\p_{i} V_{{\bf \Phi},f} \big(s,e^{(s-t)B}x\big) - \p_{i} V_{{\bf \Phi},f}(t,x)& =  \int_{s}^T
\underbrace{\int_{\R^N} \Big(\p_{i}\Phi(s,e^{(s-t)B}x;\t,\y)-\p_{i}\Phi(t,x;\t,\y)\Big)
f(\t,\y)d\y }_{=: I(\tau)} d\t  \\ &\quad -\underbrace{\int_{t}^{s}\int_{\R^N}
\p_{i}\Phi(t,x;\t,\y) f(\t,\y)d\y d\t}_{=: L}.
\end{aligned}
\end{equation}
We first bound the first integral. By applying \eqref{dercurveinteg} in the case $s-t<\tau -s$,
and \eqref{eq:est_der_first_Phi} in the case $s-t \geq \tau -s$, we obtain
\begin{equation}
I(\tau)\leq C \|f\|_{\Linf{T,\gamma}{\av}} (s-t)^{\frac{1+\av}{2}} (T-\tau)^{-\gamma}
(\tau-s)^{-1+\frac{\aa -\av}{2}},
\end{equation}
which yields
\begin{equation}\label{eq:est_I_pot_campi}
\Big| \int_{s}^T  I(\tau)  d \tau \Big| \leq C
\|f\|_{\Linf{T,\gamma}{\av}}(T-s)^{-\gamma+\frac{\aa -\av}{2}} (s-t)^{\frac{1+\av}{2}} .
\end{equation}
As for $L$, estimate \eqref{eq:est_der_first_Phi} 
yields
\begin{align}
|L| & \leq C \int_{t}^{s} \frac{1}{(T - \tau )^{\gamma} (\tau-t)^{\frac{1-\aa}{2}}} \int_{\R^N}
\G^{2\m}(t,x;\tau,\y)  d\y d\t \|f\|_{\Linf{T,\gamma}{\av}} \intertext{(by integrating in $\y$,
and since $T-s\leq T-\tau$)} & \leq C (T-s)^{-\gamma}  \int_{t}^{s} \frac{1}{
(\tau-t)^{\frac{1-\aa}{2}}} d\t \|f\|_{\Linf{T,\gamma}{\av}} \leq C (T-s)^{-\gamma } (s-t)^{\frac{1+\av}{2}} \|f\|_{\Linf{T,\gamma}{\av}},
\end{align}
which, together with \eqref{eq:est_I_pot_campi}, proves \eqref{reg_integral_Vparam}.

The proof of \eqref{reg_integral_Vparam_bis} is completely analogous, by employing
\eqref{eq:est_der_second_Phi}-\eqref{dercurveinteg2} in place of
\eqref{eq:est_der_first_Phi}-\eqref{dercurveinteg}.
\end{proof}

\subsection{H\"older estimates for 
$V_{g}$}

In this section we prove the following H\"older estimates for 
$V_{g}$, on which the proof of Theorem \ref{th:main} relies. 
\begin{proposition}\label{prop:holder_bounds_Vg}
For any $i,j,k=1,\dots,d$, we have
\begin{align}\label{eq:bound_g}
  | V_{g}(t,x)  | &\leq C \| g \|_{\Canis{\beta}}, \\
\label{eq:bound_der_first_g}
  | \p_{i} V_{g}(t,x)  | &\leq C (T-t)^{-\frac{(1 - \beta) \vee 0}{2}} \| g \|_{\Canis{\beta}}, \\
| \p_{i j} V_{g}(t,x)  |&\leq C (T-t)^{-\frac{(2 - \beta) \vee 0}{2}} \| g \|_{\Canis{\beta}},
\label{eq:bound_der_second_g}
\end{align}
and
\begin{align}\label{fcampidritti_g}
|\p_{i j} V_{g} (t,x+h\ee_k) - \p_{i j} V_{g}(t,x)  |
&\leq C |h|^{\av} (T-t)^{-\frac{2 + \av - \beta }{2}} \| g \|_{\Canis{\beta}}, 
 \\ \label{reg_integral_Vparam_g}
  |\p_{i} V_{g} \big(s,e^{(s-t)B}x\big) - \p_{i} V_{g}(t,x)  |
&\leq C(s-t)^{\frac{1+\av}{2}}(T-s)^{-\frac{2 + \av - \beta }{2}} \| g \|_{\Canis{\beta}}, 
\\ \label{reg_integral_Vparam_bis_g}
|\p_{i j} V_{g} \big(s,e^{(s-t)B}x\big) - \p_{i j} V_{g}(t,x)  |
&\leq C(s-t)^{\frac{\av}{2}} (T-s)^{-\frac{2 + \av - \beta }{2}} \| g \|_{\Canis{\beta}}, 
\end{align}
for any $0<t<s<T$, $x\in\rn$ and $h\in \R$. 
\end{proposition}
Recall that, by assumption, $g\in\Canis{\beta}$ with ${\beta \in [0,2+\av]}$. Therefore, for any
fixed $\bar x \in\rn$, {
 the following
truncated Taylor polynomials are well defined
\begin{equation}
\tilde T_{\beta,\bar x}g( y)  :=
  \begin{cases}\psi(y-\bar{x}) g(\bar x),&  \text{if } \beta\in]0,1],  \\
\psi(y-\bar{x})  \Big(  g(\bar x) +   \sum\limits_{i=1}^d \partial_{i} g(\bar x) (y_i - \bar x_i) \Big) 
, &  \text{if } \beta\in]1,2], \\
\psi(y-\bar{x})  \Big( g(\bar x) +   \sum\limits_{i=1}^d \partial_{i} g(\bar x) (y_i - \bar x_i) 
 +  \frac{1}{2} \sum\limits_{i,j=1}^d \partial_{ij} g(\bar x) (y_i - \bar x_i) (y_j - \bar x_j) \Big) 
 , &  \text{if } \beta\in]2,3],
 \end{cases}
\end{equation}
for any $y\in\rn$, where $\psi$ is a 
cut-off function such that $\psi(x) = 1$ if $|x|_B\leq 1$.
We also set the remainder
\begin{equation}
R^g_{\beta,\bar{x}}( y) := g(y) - \tilde T_{\beta,\bar{x}}g
(y) , \qquad \bar x, y \in\rn.
\end{equation}
}

The next lemma is a straightforward consequence of the definition of anisotropic norm $\| \cdot
\|_{\Canis{\beta}}$. 
\begin{lemma}[\bf Taylor formula]\label{lem:taylor_anis} We have
\begin{equation}
| R^g_{\beta,\bar{x}}( y) | \leq C \| g \|_{\Canis{\beta}}  |   y - \bar x   |^{\beta}_B, \qquad
y,\bar x \in \rn.
\end{equation}
\end{lemma}

\begin{remark}\label{rem:g_reduction}
For any $\bar x\in \rn$, {
 we have}
\begin{equation}\label{eq:estimates_taylor_holder}
\| 
{\tilde T_{\beta,\bar{x}}g } \|_{\Cint{T}{2+\av}} + \| \Lcalp {\tilde T_{\beta,\bar{x}}g}
\|_{\Linf{T}{\av}} \leq C \| g \|_{\Canis{\beta}}.
\end{equation}
Now set $u_{\bar x}(t,x):=V_{g}(t,x) - 
{\tilde T_{\beta,\bar{x}}g(x)}$ so that
\begin{equation}
V_{g}(t,x) = 
{\tilde T_{\beta,\bar{x}}g(x)} + u_{\bar x}(t,x).
\end{equation}
By the first part of Theorem \ref{th:main}, $V_{g}$ is the solution to the Cauchy problem
\ref{eq:Cauchy_prob} with $f=0$. Therefore, it is easy to check that
$u_{\bar x}$ is the solution to the Cauchy problem \ref{eq:Cauchy_prob} with
\begin{equation}
f = - \Lcalp {\tilde T_{\beta,\bar{x}}g} 
\end{equation}
and terminal datum given by $R^g_{\bar x ,\beta}$. In particular (see \eqref{eq:repres_u_gzero}), $u_{\bar x}
$ is of the form
\begin{equation}
u_{\bar x}
 = V_{R^g_{\bar x ,\beta}} - V_{\param,f}
  - V_{\Phi,f}.
\end{equation}
Therefore, owing to Proposition \ref{prop:holder_bounds_V_param_Phi} and to
\eqref{eq:estimates_taylor_holder}, in order to prove the inequalities in Proposition
\ref{prop:holder_bounds_Vg}, it is sufficient to prove them for $V_{R^g_{\bar x ,\beta}}$, with an
arbitrary $\bar x \in \rn$.
\end{remark}

\begin{proof}[Proof of Proposition \ref{prop:holder_bounds_Vg}
]

Let $0<t<s<T$, $x\in\rn$ and $h\in \R$ be fixed.
For brevity, we only prove \eqref{eq:bound_der_second_g},  \eqref{fcampidritti_g} and
\eqref{reg_integral_Vparam_bis_g}, the proofs of \eqref{eq:bound_g}, \eqref{eq:bound_der_first_g}
and \eqref{reg_integral_Vparam_g} begin simpler.


We first prove \eqref{eq:bound_der_second_g}. By Remark \ref{rem:g_reduction}, it is enough to
prove the estimate for $V_{R^g_{\xi,\beta}}$ with $\x:=e^{(T-t)B}x$.
%
By \eqref{eq:represent_deriv_potent_param_g}, which remains true for $V_{R^g_{\xi,\beta}}$, and
Lemma \ref{lem:taylor_anis}, we obtain
\begin{align}
| \p_{i j} V_{R^g_{\xi,\beta}}(t,x)  |
&\leq C \| g \|_{\Canis{\beta}} \int_{\rn} \frac{\G^{{2\m}}(t,x;\T,\y)}{\T-t} |\y-\x|_B^\b d\y
\intertext{(by the Gaussian estimates in \cite[Lemma A.6]{MR4660246}, and integrating
in $\eta$)} &\leq C \| g \|_{\Canis{\beta}} (T-t)^{-\frac{2 - \beta}{2}} .
\end{align}

We now prove \eqref{fcampidritti_g} by considering two separate cases.
\\\vspace{2pt}\emph{Case $h^2\geq T-t$.}
By employing triangular inequality, we have
\begin{align}
|\p_{i j} V_{g} (t,x+h\ee_k) - \p_{i j} V_{g}(t,x)  | &\leq 
|\p_{i j} V_{g} (t,x+h\ee_k) | + | \p_{i j} V_{g}(t,x)  | \intertext{(by
\eqref{eq:bound_der_second_g} that we have just proved)} &\leq
\frac{C}{(T-t)^{\frac{2-\b}{2}}}\|g\|_{\Canis{\beta}}
\leq C\frac{{|h|^\av}}{(T-t)^{\frac{2-\b+\av}{2}}}\|g\|_{\Canis{\beta}},
\end{align}
where we used $T-t \leq h^2$ in the last inequality.
\\\vspace{2pt}\emph{Case $h^2 < T-t$.} Once more, by Remark \ref{rem:g_reduction}, it is enough to prove the estimate for $V_{R^g_{\xi,\beta}}$ with $\x:=e^{(T-t)B}x$. First we note that, 
setting $\x':=e^{(T-t)B}(x+h\ee_k)$, \cite[Lemma A.4]{MR4660246} yields
\begin{align}
|\x'-\x|_B=&|h e^{(T-t)B} \ee_k |_B \leq C \sum_{j=0}^q |h (T-t)^{j}|^{\frac{1}{2j+1} } \leq C
(T-t)^{\frac{1}{2}},
\end{align}
and thus
\begin{align}
\G^{2\m}(t,x +h \mathbf{e}_k ;T,\y) |\y-\x|_B^\b 
& \leq \G^{2\m}(t,x +h \mathbf{e}_k ;T,\y)(|\y-\x'|_B^\b+|\x'-\x|_B^\b ) \intertext{(by
\cite[Lemma A.6]{MR4660246})} &\leq C(T-t)^{\frac{\b}{2}} \G^{3\m}(t,x +h \mathbf{e}_k
;T,\y) + \G^{2\m}(t,x +h \mathbf{e}_k ;T,\y) |\x'-\x|_B^\b \\
&\leq C(T-t)^{\frac{\b}{2}} \G^{3\m}(t,x +h \mathbf{e}_k ;T,\y). \label{eq:estimatexip0}
\end{align}
By Proposition \ref{prop:repr_der_poten_g} and Lemma \ref{lem:taylor_anis}, we obtain
\begin{align}
|\p_{i j} V_{R^g_{\xi,\beta}}(t,x+h\ee_k) - \p_{i j} V_{R^g_{\xi,\beta}}(t,x) | &\leq C  \| g \|_{\Canis{\beta}} \int_{\R^N} |\p_{i j}\sol(t,x+h\ee_k;T,\y)-\p_{i j}\sol(t,x;T,\y)| 
|\y-\x|_B^\b d\y 
\intertext{(by \eqref{dercampidritti_param}-\eqref{dercampidritti})} &\leq C\| g
\|_{\Canis{\beta}} |h|^{{\av}} \int_{\R^N}  \frac{\G^{2\m}(t,x +h \mathbf{e}_k
;T,\y)+\G^{2\m}(t,x;T,\y)}{(T-t)^{\frac{2+ \av}{2}}} |\y-\x|_B^\b d\y \\
\intertext{(by the estimate \eqref{eq:estimatexip0} and, once more, the Gaussian estimates in
\cite[Lemma A.6]{MR4660246})}
&\leq C  \| g \|_{\Canis{\beta}} |h|^{{\av}}\int_{\R^N}  \frac{\G^{3\m}(t,x +h \mathbf{e}_k ;T,\y)
+ \G^{2\m}(t,x;T,\y)}{(T-t)^{\frac{2+ \av-\b}{2}}}d\y\\ \intertext{{(integrating in $\y$)}} &{\leq
C} \| g \|_{\Canis{\beta}} {|h|^{\av}} {(T-t)^{\frac{-2-\av+\b}{2}}}.
\end{align}

We finally prove \eqref{reg_integral_Vparam_bis_g} 
by considering two separate cases.
\\ \vspace{2pt}\emph{Case $s-t\geq T-s$.}
By employing triangular inequality, we have
\begin{align}
|\p_{i j} V_{g} \big(s,e^{(s-t)B}x\big) - \p_{i j} V_{g}(t,x)  | & \leq |\p_{i j} V_{g}
\big(s,e^{(s-t)B}x\big) |+| \p_{i j} V_{g}(t,x)  | \intertext{(by \eqref{eq:bound_der_second_g})}
&\leq {C} \|g\|_{\Canis{\beta}} \big(  (T-t)^{\frac{-2+\b}{2}} + (T-s)^{\frac{-2+\b}{2}}  \big)
\leq C \|g\|_{\Canis{\beta}} {(s-t)^\frac{\av}{2}} (T-t)^{\frac{-2-\av+\beta}{2}},
\end{align}
where we used $ T-s \leq s-t \leq T-t $ in the last inequality.
\\ \vspace{2pt}\emph{Case $s-t < T-s$.} Once more, by Remark \ref{rem:g_reduction}, it is enough to prove the estimate for $V_{R^g_{\xi,\beta}}$ with $\x:=e^{(T-t)B}x$. By Proposition \ref{prop:repr_der_poten_g} and Lemma \ref{lem:taylor_anis} we obtain
\begin{align}
|\p_{i j} V_{R^g_{\xi,\beta}} \big(s,e^{(s-t)B}x\big) - \p_{i j} V_{R^g_{\xi,\beta}}(t,x)| & \leq
C \|g\|_{\Canis{\beta}} \int_{\rn} |\p_{ij}\sol(s,e^{(s-t)B}x;T,\y) - \p_{ij}\sol(t,x;T,\y)|
|\y-\x|_B^\b d \y \intertext{(by \eqref{dercurveinteg2_param}-\eqref{dercurveinteg2}, and since
$s-t<T-s$)}
&\leq C \| g \|_{\Canis{\beta}}\frac{(s-t)^{\frac{\av}{2}} }{(T-s)^{1+\frac{\av}{2}}} \int_{\R^N}
\G^{2\m}(t,x;T,\y) |\y-\x|_B^\b d\y \intertext{(by the Gaussian estimates in \cite[Lemma
A.6]{MR4660246}, and integrating in $d \y$)} & \leq C \| g
\|_{\Canis{\beta}}(T-s)^{\frac{-1-\av+\beta}{2}} (s-t)^{\frac{\av}{2}} .
\end{align}
\end{proof}

\bibliographystyle{siam}
\bibliography{BibTeX-Final}

%

\end{document}